\documentclass[a4paper]{mathscan}
 
       \newtheorem{thm}{Theorem}[section] %%the counter [section] is optional  
       \newtheorem{cor}[thm]{Corollary}    
       \newtheorem{lem}[thm]{Lemma}        

       \theoremstyle{definition} %%%%% switch to Roman bodyfont
         
       \newtheorem{rem}[thm]{Remark}   %%%%%the counter [thm] is optional here 
\newcommand{\supp}{\operatorname{supp}}
\newcommand{\rp}{ \mathbb R_+}
\newcommand{\crp}{\overline{\mathbb R}_+}

\newcommand{\rn}{{\mathbb R}^n}
\newcommand{\rnp}{{\mathbb R}^n_+}

\newcommand{\rnpm}{\mathbb R^n_\pm}
\newcommand{\crnp}{\overline{\mathbb R}^n_+}

\newcommand{\crnpm}{\overline{\mathbb R}^n_\pm}
\newcommand{\comega}{\overline\Omega }
\newcommand{\ang}[1]{\langle {#1} \rangle}
\newcommand{\Op}{\operatorname{Op}}
\newcommand{\simto}{\overset\sim\rightarrow}
\newcommand{\ol}{\overline}

\newcommand{\R}{\mathbb R}

\newcommand{\SD}{\mathcal S}

\newcommand{\F}{\mathcal F}
\newcommand{\D}{\mathcal D}

\begin{document}

\title [Exact Green's formula for the fractional Laplacian]
{Exact Green's formula for the fractional Laplacian  and perturbations}

\author {Gerd Grubb}

\address
{Dept. of Mathematical Sciences,\\ Copenhagen University,\\
Universitetsparken 5,\\ DK-2100 Copenhagen, Denmark.\\
E-mail {\tt grubb\@math.ku.dk}}
\begin{abstract}
Let $\Omega $ be an open, smooth, bounded subset of $ \rn$. In connection with
the fractional
Laplacian $(-\Delta )^a$ ($a>0$), and more generally for a $2a$-order
classical pseudodifferential operator  ($\psi $do) $P$ with even
symbol, one can define the Dirichlet value $\gamma _0^{a-1}u$ resp.\
Neumann value $\gamma _1^{a-1}u$ of $u(x)$ as the trace resp.\ normal
derivative of $u/d^{a-1}$ on $\partial\Omega $, where $d(x)$ is the
distance from $x\in\Omega $ to $\partial\Omega $; they define
well-posed boundary value problems for $P$. 

A Green's formula was
shown in a preceding paper, containing a generally nonlocal term
$(B\gamma _0^{a-1}u,\gamma _0^{a-1}v)_{\partial\Omega }$, where $B$ is a 
first-order $\psi $do
on $\partial\Omega $.
Presently, we determine $B$ from $L$ in the case $P=L^a$, where   $L$ is  a strongly elliptic
second-order differential operator. A particular result is that $B=0$ when
$L=-\Delta $, and that $B$ is multiplication by a function (is local)
when $L$ equals $-\Delta $ plus a first-order term. In cases of more general
$L$,  $B$ can be nonlocal.
\end{abstract}


\maketitle

\section { Introduction}\label{sec1}

The fractional Laplacian $(-\Delta )^a$ on ${\mathbb R}^n$, $a>0$, is
 currently
receiving much attention because of its great interest for applications in
both probability, finance, mathematical physics and differential
geometry. (References to many important contributions through the
years are given e.g.\ in our preceding papers \cite{G14}--\cite{G16b}.)
$(-\Delta )^a$ can be defined as a pseudodifferential operator ($\psi $do), or equivalently
 as a singular integral operator:
\begin{align}\label{eq:1.1}
(-\Delta )^au&=\operatorname{Op}(|\xi |^{2a})u=
\mathcal F^{-1}(|\xi |^{2a}\hat u(\xi )),\\ \nonumber
(-\Delta )^au(x)&=c_{n,a}PV\int_{{\mathbb
R}^n}\frac{u(x)-u(y)}{|x-y|^{n+2a}}
\,dy,
\end{align}
where $\mathcal F$ denotes Fourier transformation $\hat
u(\xi )=\mathcal F
u= \int_{{\mathbb R}^n}e^{-ix\cdot \xi }u(x)\, dx$.
Since the operator is nonlocal for noninteger $a$, it is not obvious
how to define its action over a subset $\Omega $ of ${\mathbb
R}^n$, and there are several ways to define operators on $\Omega $
representing homogeneous boundary value problems for it (see e.g.\ the
overview
in Section 6 of \cite{G16a}).

A much studied case is the so-called restricted Dirichlet problem
\begin{equation}\label{eq:1.2}
%r^+(-\Delta )^au\equiv 
  ((-\Delta )^au)|_{\Omega }=f,\quad \supp u\subset\comega,
  \end{equation}
considered for functions $u$ and $f$ with a certain regularity. 

One can also impose {\emph nonhomogeneous} boundary conditions. We
are particularly interested in {\emph local} boundary operators (i.e,\
operators that
can be defined pointwise at $\partial\Omega $). It was shown in
\cite{G15}, Sect.\ 5, for smooth open sets $\Omega $, that the local  operators \begin{equation}
\gamma _j^{a+k}u=c_{akj}\gamma _j(u/d^{a+k}), \; j\in{\mathbb N}_0,\; k\in {\mathbb
Z}\text{ with }a+k>-1,\label{eq:1.3}
\end{equation} 
$c_{akj}=\Gamma(a+k+j+1)$, have a meaning in connection with $(-\Delta )^a$; here
$d=\operatorname{dist}(x,\partial\Omega )$, and $\gamma _jv$ is the
standard normal derivative $(\partial_\nu ^jv)|_{\partial\Omega }$. In
particular, defining the
\begin{align}
\label{eq:1.4}  \text{ Dirichlet
trace }\gamma_0^{a-1}u&=\Gamma (a)\gamma _0(u/d^{a-1}),\\ 
\text{ Neumann
 trace } \gamma _1^{a-1}u&=\Gamma (a+1)\gamma _1(u/d^{a-1}), 
\nonumber
\end{align} 
 one obtains
well-posed nonhomogeneous  boundary value problems on $\Omega $ for
$(-\Delta )^a$ 
and more general operators; see
\cite{G15} for the Dirichlet condition, and \cite{G14,G18} for the Neumann
condition. The solutions are found to lie in so-called $\mu $-transmission
spaces (recalled in Section 2 below) with $\mu =a-1$ or $a$.

%[Hold \o{}je med hvor $a<1$ indg\aa{}r.]

When $0<a<1$, the solutions $u$ with nonzero $\gamma _0^{a-1}u$ have
an unbounded
singularity like $d^{a-1}$ at the boundary (also studied in Abatangelo
\cite{A15}). However if $\gamma _0^{a-1}u=0$,
$u$ behaves like $d^a$ at the boundary, and $\gamma _1^{a-1}u$ coincides with
$\gamma _0^{a}u$. 

 Recently,  Abatangelo, Jarohs and Saldana in
\cite{AJS18},  with further coauthors in
\cite{AD$^+$19}, have studied nonhomogeneous boundary value problems for $(-\Delta )^a$ involving the trace operators
\eqref{eq:1.3}, on the unit ball resp.\ halfspace in $\rn$, with detailed calculations. 

\medskip
Formulas for integration by parts were first shown for functions with
$\gamma _0^{a-1}u=0$ by Ros-Oton and
Serra \cite{RS14,RS15} 
(and jointly with Valdinoci for more general
singular integral operators \cite{RSV17}) and Abatangelo \cite{A15}, leading to Pohozaev identities important for
uniqueness questions in nonlinear applications.
 In  \cite{G16b}, we extended the formulas to general $x$-dependent
$2a$-order pseudodifferential operators $P$ satisfying the
$a$-transmission 
condition at $\partial\Omega $.
 
More recently in \cite{G18} we obtained a general Green's formula for
functions $u,v$ in $(a-1)$-transmission spaces, allowing
both $\gamma _0^{a-1}u$ and $ \gamma _1^{a-1}u$ to be nonzero:
 \begin{equation}(Pu,v)_\Omega -(u,P^*v)_\Omega =(s_0\gamma _1^{a-1}u+B\gamma
_0^{a-1}u,\gamma _0^{a-1}v)_{\partial\Omega }-(s_0\gamma
_0^{a-1}u,\gamma _1^{a-1}v)_{\partial\Omega }.\label{eq:1.5}
\end{equation}
% where $\gamma _k^{a-1}u=\Gamma (a+k)\partial_n^k(u/d^{a-1})|_{\partial\Omega }$;
Here $s_0(x)$ is a function defined from the principal symbol of $P$, and $B$ is a first-order $\psi $do on $\partial\Omega $ depending on
the first two terms in the symbol of $P$. 
%The operator $B$ will
It is nonlocal in general. 

%\medskip

In the present paper we investigate how $B$ looks in particular
cases. We show that for  $(-\Delta )^a$ itself, the operator $B$ is zero:
\begin{equation}
((-\Delta )^au,v)_\Omega -(u,(-\Delta )^av)_\Omega =(\gamma _1^{a-1}u,\gamma _0^{a-1}v)_{\partial\Omega }-(\gamma
_0^{a-1}u,\gamma _1^{a-1}v)_{\partial\Omega },\label{eq:1.6}
\end{equation}
and for operators $(-\Delta +\pmb c(x)\cdot \nabla+c_0(x))^a$, $B$ is the
multiplication by a function derived from $\pmb c$ (Theorem
\ref{Theorem 5.1}). In
these cases, $B$ is local.

More generally, 
we investigate powers  $L^a$ of a general second-order strongly elliptic partial differential
operator $L$, finding formulas for $B$ in local coordinates
(Theorem \ref{Theorem 4.2}). It is seen here that when the normal component of the principal part of $L$
varies along $\partial\Omega $, $B$ can be nonlocal (Remark
\ref{Remark 4.3}).

\medskip
\noindent \emph {Plan of the paper:} In Section 2 we list some prerequisites and
recall the definition and properties of the $\mu $-transmission spaces that play an
important role as domains.
%; they are further analysed. 
In Section 3 we
find the symbol of the fractional power $L^a$ with two leading terms,
when $L$ is a strongly elliptic differential operator $-{\sum}_{j,k\le
n}a_{jk}\partial_j\partial_k+\pmb b\cdot \nabla +b_0$. In Section 4 we
determine the contribution from $L^a$ to the symbol of $B$, in the case
$\Omega =\rnp$. In Section 5 we apply this to the case of general
smooth bounded sets $\Omega $ when 
$L$ has principal part  $-\Delta $,
%in the general curved situation, 
showing that $B$ is the
multiplication by a certain function, which vanishes when the first-order
part 
%term $\pmb b\cdot\nabla$ 
is zero. The Appendix gives an analysis of
Green's formula for $-\Delta $, connecting the formula for the general
set $\Omega $ with the localized case and providing some
ingredients for the treatment of $(-\Delta )^a$. Some misprints in
\cite{G18} are listed at the end. 

\section {Notation and preliminaries, the $\mu $-transmission spaces}\label{sec2} 

Our notation has already been explained in several preceding papers
\cite{G14}--\cite{G18}, so we shall only recall the most important concepts
needed here.

Multi-index notation is used for
differentiation (and also for polynomials): 
$\partial=(\partial_1,\dots,\partial_n)$, and $\partial^\alpha
=\partial_1^{\alpha _1}\dots \partial_n^{\alpha _n}$ for $\alpha
\in{\mathbb N}_0^n$, with $|\alpha |=\alpha _1+\cdots+\alpha _n$, $\alpha
!=\alpha _1!\dots\alpha _n!$. $D=(D_1,\dots,D_n)$ with $D_j=-i\partial_j$.
The function $\ang\xi $ stands for $(1+|\xi |^2)^\frac12 $.

Operators are considered acting on functions or distributions on
${\mathbb R}^n$, and on subsets  
 $\rnpm=\{x\in
{\mathbb R}^n\mid x_n\gtrless 0\}$ (where $(x_1,\dots, x_{n-1})=x'$), and
 bounded $C^\infty $-subsets $\Omega $ with  boundary $\partial\Omega $, and
their complements.
Restriction from $\R^n$ to $\rnpm$ (or from
${\mathbb R}^n$ to $\Omega $ resp.\ $\complement\comega$) is denoted $r^\pm$,
 extension by zero from $\rnpm$ to $\R^n$ (or from $\Omega $ resp.\
 $\complement\comega$ to ${\mathbb R}^n$) is denoted $e^\pm$. Restriction
 from $\crnp$ or $\comega$ to $\partial\rnp$ resp.\ $\partial\Omega $
 is denoted $\gamma _0$.

We denote by $d(x)$ a function of the form $
d(x)=\operatorname{dist}(x,\partial\Omega )$ for $x\in\Omega $, $x$ near $\partial\Omega $,
extended to a smooth positive function on $\Omega $; $d(x)=x_n$ in the
case of $\rnp$. Then we define the spaces
\begin{equation}
\mathcal E_\mu (\comega)=e^+\{u(x)=d(x)^\mu v(x)\mid v\in C^\infty
(\comega)\},\label{eq:2.1}
\end{equation}
for $\operatorname{Re}\mu
 >-1$; for other $\mu $, cf.\ \cite{G15}.

A \emph {pseudodifferential operator} ($\psi $do) $P$ on ${\mathbb R}^n$ is
defined from a symbol $p(x,\xi )$ on ${\mathbb R}^n\times{\mathbb R}^n$ by 
\begin{equation}
Pu=p(x,D)u=\operatorname{Op}(p(x,\xi ))u 
=(2\pi )^{-n}\int e^{ix\cdot\xi
}p(x,\xi )\hat u\, d\xi =\mathcal F^{-1}_{\xi \to x}(p(x,\xi )\hat u(\xi
)),\label{eq:2.2}
\end{equation}  
using the Fourier transform $\F$, cf.\ \eqref{eq:1.1}ff. 
%$=\hat u(\xi
%)=\int_{{\mathbb R}^n}e^{-ix\cdot \xi }u(x)\, dx$. 
We refer to
textbooks such as H\"ormander \cite{H85}, Taylor \cite{T81}, Grubb \cite{G09} for the rules of
calculus. 
$p$ belongs to the symbol space $S^m_{1,0}({\mathbb R}^n\times{\mathbb R}^n)$, consisting of
$C^\infty $-functions $p(x,\xi )$
such that $\partial_x^\beta \partial_\xi ^\alpha p(x,\xi
)$ is $O(\ang\xi ^{m-|\alpha |})$ for all $\alpha ,\beta $, for some
$m\in{\mathbb R}$ (global estimates); then $P$ (and $p$) has order $m$.
$P$ (and $p$) is said to be \emph {classical} when $p$ moreover 
has an asymptotic expansion $p(x,\xi )\sim \sum_{j\in{\mathbb
N}_0}p_j(x,\xi )$ with $p_j$ homogeneous in $\xi $ of degree $m-j$ for
$|\xi |\ge 1$, all $j$, and $p(x,\xi )- \sum_{j<J}p_j(x,\xi )\in S^{m-J}_{1,0}({\mathbb
R}^n\times \R^n)$ for all $J$. 
%$P$ is then said to be \emph {elliptic}
%when $p_0(x,\xi )\ne 0$ for $|\xi |\ge 1$.

Recall in particular the composition rule: When $PQ=R$, then $R$ has
a symbol $r(x,\xi )$ with the following asymptotic expansion, called the
Leibniz product:
\begin{equation}
r(x,\xi )\sim p(x,\xi )\# q(x,\xi )= \sum_{\alpha \in{\mathbb
N}_0^n}\partial_\xi ^\alpha p(x,\xi ) D_x^aq(x,\xi )/\alpha !.\label{eq:2.3}
\end{equation}

When  $P$ (and $p$) is classical, it is said to be \emph
{even}, when
\begin{equation}
p_j(x,-\xi )=(-1)^jp_j(x,\xi ),\text{ all }j. \label{eq:2.4}
\end{equation}
Then if $P$ is of order $2a$, it satisfies the $a$-transmision condition defined in \cite{G15},
with respect to any smooth subset $\Omega $ of ${\mathbb
R}^n$. Even-order differential operators $L$ have this  evenness property,
and so do the powers $L^a$ (as constructed by Seeley \cite{S67}) when $L$ is strongly elliptic.

When $P$ is a $\psi $do on ${\mathbb R}^n$, $P_+=r^+Pe^+$
denotes its truncation to $\rnp$, or to $\Omega $, depending on the context.
%\medskip

%Let $1<p<\infty $  (with $1/p'=1-1/p$), then 
The $L_2$-Sobolev spaces are defined for $s\in{\mathbb R}$ by
\begin{align*}
H^s(\R^n)&=\{u\in \SD'({\mathbb R}^n)\mid \F^{-1}(\ang{\xi }^s\hat u)\in
L_2(\R^n)\},\\
\dot H^{s}(\comega)&=\{u\in H^{s}({\mathbb R}^n)\mid \supp u\subset
\comega \}, \text{ the supported space}\\
\ol H^{s}(\Omega)&=\{u\in \D'(\Omega )\mid u=r^+U \text{ for a }U\in
H^{s}(\R^n)\}, \text{ the restricted space};
\end{align*} 
here $\operatorname{supp}u$ denotes the support of $u$. The definition
is also used with $\Omega =\rnp$. In most current texts, $\ol
H^s(\Omega )$ is denoted $H^s(\Omega )$ without the overline (that
was introduced along with the notation $\dot H$ in \cite{H66,H85}), but we keep it here since it is
practical in indications of dualities, and makes the notation more
clear in formulas where
both types occur. 
%When $p=2$, the mention of $p$ is left out.
We recall that $\ol H^s(\Omega )$ and $\dot H^{-s}(\comega)$ are dual
spaces with respect to a sesquilinear duality extending the $L_2(\Omega )$-scalar
product, written e.g.\
\begin{equation*}
\ang{f,g}_{\ol H^s(\Omega ),\dot H^{-s}(\comega)},\text{ or just }\ang{f,g}_{\ol H^s,\dot H^{-s}}.
\end{equation*}

There are many other interesting scales of spaces, the
Bessel-potential spaces $H^s_p$,
the
Triebel-Lizorkin  spaces $F^s_{p,q}$ and the Besov spaces $B^s_p$ and $B^s_{p,q}$, where
the problems can be studied; see details in \cite{G14,G15}. This includes the H\"older-Zygmund spaces $B^s_{\infty
,\infty }$, also denoted $C^s_*$; they are interesting because $C^s_*({\mathbb R}^n)$ equals
the H\"older space $C^s({\mathbb R}^n)$ when $s\in\rp\setminus {\mathbb
N}$. The survey in \cite{G18a} Sect.\ 3 recalls the theory in
$H^s_p$-spaces. We continue here with $p=2$.

A special role in the theory is played by the \emph{order-reducing
operators}. There is a simple definition of operators $\Xi _\pm^t $ on
${\mathbb R}^n$ for $t\in{\mathbb R}$,
\begin{equation} 
\Xi _\pm^t =\operatorname{Op}(\chi _\pm^t),\quad \chi _\pm^t=(\ang{\xi '}\pm i\xi _n)^t ;\label{eq:2.5} 
\end{equation}
 they preserve support
in $\crnpm$, respectively. The functions
$(\ang{\xi '}\pm i\xi _n)^t $ do not satisfy all the estimates
required for the class $S^{t }_{1,0}({\mathbb
R}^n\times{\mathbb R}^n)$, but the operators are useful for many
purposes. There is a more refined choice $\Lambda _\pm^t $
\cite{G90, G15}, with
symbols $\lambda _\pm^t (\xi )$ that do
satisfy all the estimates for $S^{ t }_{1,0}({\mathbb
R}^n\times{\mathbb R}^n)$; here $\overline{\lambda _+^t }=\lambda _-^{t }$.
The symbols have holomorphic extensions in $\xi _n$ to the complex
halfspaces ${\mathbb C}_{\mp}=\{z\in{\mathbb C}\mid
\operatorname{Im}z\lessgtr 0\}$; it is for this reason that the operators preserve
support in $\crnpm$, respectively. Operators with that property are
called "plus" resp.\ "minus" operators. There is also a pseudodifferential definition $\Lambda
_\pm^{(t )}$ adapted to the situation of a smooth domain $\Omega
$, cf.\ \cite{G15}.

It is elementary to see by the definition of the spaces $H^s(\R^n)$
in terms of Fourier transformation, that the operators define homeomorphisms 
for all $s$: $
\Xi^t _\pm\colon H^s(\R^n) \simto H^{s- t
}(\R^n)$, $  
\Lambda ^t _\pm\colon H^s (\R^n) \simto H^{s- t
} (\R^n)$.
%(and so does of course $\Xi ^t =\Op(\ang \xi ^t )=\ang D^t$). 
The special
interest is that the "plus"/"minus" operators also 
 define
homeomorphisms related to $\crnp$ and $\comega$, for all $s\in{\mathbb R}$: 
$
\Xi ^{t }_+\colon \dot H^s(\crnp )\simto
\dot H^{s- t }(\crnp)$, $
r^+\Xi ^{t }_{-}e^+\colon \ol H^s(\rnp )\simto
\ol H^{s- t } (\rnp )$, with similar statements for $
\Lambda^{(t )}_\pm$ relative to $\Omega $.
Moreover, the operators $\Xi ^t _{+}$ and $r^+\Xi ^{t }_{-}e^+$ identify with each other's adjoints
over $\crnp$, because of the support preserving properties;
there is a
similar statement for $\Lambda ^{(t )}_+$ and $r^+\Lambda ^{(
t )}_{-}e^+$ relative to the set $\Omega $.

The special  $\mu $-\emph{transmission spaces} were 
introduced by
H\"ormander \cite{H66} and redefined in \cite{G15} (we just recall them for real $\mu >-1$):
\begin{align}\label{eq:2.6}
H^{\mu (s)}(\crnp)&=\Xi _+^{-\mu }e^+\ol H^{s- \mu
}(\rnp)=\Lambda  _+^{-\mu }e^+\ol H^{s- \mu
}(\rnp)
,\quad  s> \mu -\tfrac12,\\ \nonumber
H^{\mu (s)}(\comega)&=\Lambda  _+^{(-\mu )}e^+\ol H^{s- \mu
}(\Omega ),\quad  s> \mu -\tfrac12;
\end{align}
%\end{equation}
they are the appropriate solution spaces for homogeneous Dirichlet
problems for elliptic operators $P$ having the $\mu
$-transmission property (cf.\ \cite{G15}).
We also recall 
that $r^+P$  maps $\mathcal E_\mu
 (\comega)$ (cf.\ \eqref{eq:2.1}) into 
$C^\infty (\comega)$, and that $\mathcal E_\mu (\comega)$ is the solution
space for the homogeneous
Dirichlet problem with data in $C^\infty (\comega)$.  $\mathcal E_\mu (\comega)$ is dense in  $H^{\mu
(s)}(\comega)$ for all $s$, and $\bigcap_s H^{\mu
(s)}(\comega)=\mathcal E_\mu (\comega)$. (For $\Omega =\rnp$, $\mathcal
E_\mu (\crnp)\cap \mathcal E'({\mathbb R}^n)$ is dense in $H^{\mu
(s)}(\crnp)$ for all $s$.)

One has that $H^{\mu
 (s)}(\comega)\supset \dot H^s(\comega)$, and the elements are
 locally in $H^s$   on $\Omega $, but at the boundary they in general have a
 singular behavior (cf.\ \cite{G15} Th.\ 5.4):
\begin{equation}
H^{\mu (s)}(\comega)\begin{cases} =\dot H^s(\comega) \text{ if }s\in
\,]\mu -\tfrac12,\mu +\tfrac12[\,,\\
\subset \dot H^{s}(\comega)+e^+d^\mu  \ol H^{s- \mu }(\Omega)
\text{ if }s>\mu +\tfrac12, s-\mu -\frac12\notin {\mathbb N}.
\end{cases}
\label{eq:2.7}
\end{equation}

The inclusion in the second line of \eqref{eq:2.7} has recently been sharpened
in \cite{G19} to a precise description: When $ s \in \,]\mu +M
-\tfrac12,\mu +M +\tfrac12[\,$, $M\in{\mathbb N}$, then
\begin{equation}
H^{\mu (s)}(\comega)= \dot H^{s}(\comega)+\widetilde{\mathcal K}^\mu _M
{\prod}_{j=0}^{M-1}H^{s- \mu -j -\frac12}(\partial\Omega),
\label{eq:2.8}
\end{equation}
where $\widetilde{\mathcal K}^\mu _M$ is $d^\mu $ times a system of
 Poisson operators in the Boutet de Monvel calculus constructed in a
 simple way from a Poisson operator $K_{(0)}$ solving the Dirichlet
 problem for $-\Delta $. For $M=1$,  $\widetilde{\mathcal K}^\mu _1$ is
 proportional to $d^\mu K_{(0)}$.

Analogous results hold in the other scales of function spaces
($H^s_p$, $B^s_p$, $F^s_{p,q}$, $B^s_{p,q}$) mentioned above. Let us
 in particular mention the H\"older-Zygmund spaces
$C^s_*=B^s_{\infty ,\infty }$ (coinciding with ordinary H\"older
spaces for $s\in \rp\setminus{\mathbb N}$). Here the $\mu $-transmission
spaces are defined by 
\begin{align}\label{eq:2.9}
C_*^{\mu (s)}(\crnp)&=\Xi _+^{-\mu }e^+\ol C_*^{s- \mu
}(\rnp)=\Lambda  _+^{-\mu }e^+\ol C_*^{s- \mu
}(\rnp)
,\quad  s> \mu -1,\\ \nonumber
C_*^{\mu (s)}(\comega)&=\Lambda  _+^{(-\mu )}e^+\ol C_*^{s- \mu
}(\Omega ),\quad  s> \mu -1.
\end{align}
Again, $C_*^{\mu
 (s)}(\comega)\supset \dot C_*^s(\comega)$, and the elements are
 locally in $C_*^s$   on $\Omega $. More precisely,
$C_*^{\mu (s)}(\comega) =\dot C_*^s(\comega) \text{ if }s\in
\,]\mu -1,\mu [\,$, and 
 when $s 
\in \,]\mu +M-1,\mu +M [\,$ for an $M\in{\mathbb N}$:
\begin{equation}
C_*^{\mu (s)}(\comega)= \dot C_*^{s}(\comega)+
\widetilde{\mathcal K}^\mu _M
{\prod}_{j=0}^{M-1}C_*^{s- \mu -j }(\partial\Omega)
\subset \dot C_*^{s}(\comega)+ e^+d^\mu  \ol C_*^{s- \mu }(\Omega),
\label{eq:2.10}
\end{equation}
cf.\ \cite{G14,G19}. The spaces $C^s_*$ are denoted $\Lambda _s$ in
Stein \cite{S70} and sequels.

In the present paper, we shall in particular work with the spaces where $\mu =a-1$, which is
negative in the important case where $0<a<1$. The results in cases where $a>1$, for
example for $(-\Delta )^{3/2}=|\nabla |^3$, should also be of
interest.
 
 Note that we always have $\mathcal
E_{a-1}(\comega)$ as a dense subset.

\section { Powers of a second-order elliptic  differential operator }\label{sec3}

The following result was shown in \cite{G18}:

\begin{thm}\label{Theorem 3.1}  Let $P$ be a classical
$\psi $do on ${\mathbb R}^n$ of order $2a>0$ 
 (not
necessarity elliptic), with even symbol, cf.\ {\rm \eqref{eq:2.4}}, and let
$\Omega $ equal $\rnp$ or a smooth bounded subset of ${\mathbb R}^n$. The
following Green's formula holds for $u,v\in H^{(a-1)(s)}(\comega)$
when  $s>a+\frac12$ with $s\ge 2a$:
\begin{equation*}
%\aligned
\int_{\Omega }(Pu\bar v-u\overline {P^*v})\, dx=
(s_0\gamma _1^{a-1}u+B\gamma
_0^{a-1}u,\gamma _0^{a-1}v)_{L_2(\partial\Omega )}-(s_0\gamma
_0^{a-1}u,\gamma _1^{a-1}v)_{L_2(\partial\Omega )}.
\end{equation*}
(When only $s>a+\frac12$, the formula holds with the left-hand side
interpreted as dualities.)
Here  $s_0(x)=p_0(x,\nu (x))$ at boundary points $x$ with interior
normal $\nu (x)$, and $B$ is a first-order $\psi $do on $\partial\Omega $. In the case
$\Omega =\rnp$, the symbol of $B$ equals the jump at $z_n=0$ of the
distribution $\mathcal F^{-1}_{\xi _n\to z_n}q(x',0,\xi )$, where $q$ is the
symbol of $Q=\Xi ^{-a}_-P\Xi ^{-a}_+$ (the case of curved $\Omega $ is
derived from this). 

\end{thm}

Since $C^s_*\subset H^{s }$, the formula is in particular
valid when $u,v\in C^{(a-1)(s)}_*(\comega)$ for some $s>a+\frac12$ with
$s> 2a$; then $Pu$ and $P^*v$ are continuous functions on $\comega$. 

We now want to describe $B$ more precisely in interesting special cases.
A natural class of operators $P$ satifying the hypotheses arises from
taking $a$'th powers of second-order differential operators; it will
be studied in the following.

Consider a general second-order strongly  elliptic partial
differential operator given on ${\mathbb R}^n$ or on an open subset containing
the set $\comega$ we are interested in,
\begin{align} \label{eq:3.1}
L&=-{\sum}_{j,k\le n}a_{jk}\partial_j\partial_k +\pmb b(x)\cdot \nabla
+b_0(x)=L_0+L_1+L_2,
\text{ with symbols} \\
\ell&=\ell_0+\ell_1+\ell_2,
\ell_0(\xi  )={\sum}_{j,k\le n}a_{jk}\xi _j\xi_k,\quad
\ell_1(x,\xi )= \pmb b(x)\cdot i\xi, 
%ic_1(x)\xi  _1+\dots + ic_n(x)\xi _n, 
\quad l_2(x)=b_0(x),\nonumber
\end{align}
where the $a_{jk}(x)$, $\pmb b(x)=(b_1(x),\dots,b_n(x))$ and  $b_0(x)$
are bounded complex $C^\infty $-functions. The strong ellipticity means that 
\begin{equation*}
\operatorname{Re}{\sum}_{j,k\le n} a_{jk}(x)\xi _j\xi _k\ge c|\xi |^2,
\text{ for all }\xi \in{\mathbb R}^n,
\end{equation*}
with $c>0$.

We can describe the fractional powers by use of Seeley's analysis
\cite{S67}.
Assume that the functions $a_{jk}(x),b_j(x),b_0(x)$ have been extended to all of ${\mathbb
R}^n$, such that $L$ equals $1-\Delta $ outside a large
ball. The resolvent of $L$ is the
inverse of $L-\lambda $, defined when $\lambda $ is in the resolvent
set; it includes a truncated sector
$V=\{\lambda \in{\mathbb C}\mid |\arg
\lambda -\pi  |\le\frac{\pi }2+\delta \,,|\lambda |\ge R\}$ for some large
$R$
and small $\delta $. If the matrix $(a_{j,k})_{j,k\le n}$ is real (or
hermitian symmetric), $V$ can be taken as
$\{\lambda \in{\mathbb C}\mid |\arg
\lambda -\pi |\le \pi -\delta \,,|\lambda |\ge R\}$ for some large $R$
and small $\delta $. The resolvent
symbol $\tilde \ell_\lambda $ is constructed by use of the Leibniz
product formula \eqref{eq:2.3} from the symbol
$\ell-\lambda $ of $L-\lambda $.

It is known from \cite{S67} that the resolvent symbol has an expansion
in symbols $\tilde \ell_{l,\xi }$ homogeneous of degree $-2-l$ in $(\xi ,|\lambda |^{1/2})$,
\begin{align}\label{eq:3.2} 
\tilde \ell_\lambda &\sim \sum_{l=0,1,2,\dots}   \tilde\ell_{l,\lambda
},\text{ with}\\  \nonumber
\tilde\ell_{0,\lambda }&=(\ell_0-\lambda )^{-1},\quad \tilde\ell_{l,\lambda }=\sum_{ l/2\le k\le 2l} {
c_{l,k}(x,\xi )}\tilde\ell_{0,\lambda }^{-k-1}\text{ for }l=1,2,\dots, 
\end{align}
the $c_{l,k}(x,\xi ) $ being polynomials in $\xi $ of degree $2k-l$.

Let us work out the construction in exact form up to the second
homogeneous term (homogeneous of degree  $-3$ with respect to $(\xi ,\mu )$, $\mu =(-\lambda
)^{1/2}$), with the subsequent terms grouped together under the indication $l.o.t.$ (lower order
terms). We use $l.o.t.$ to denote terms of order at least two integers lower
than the principal term, in each step in the deduction (this precision
is all we need for the discussion of Green's formula).

The principal term in the resolvent symbol is $\tilde \ell_{0,\lambda
}=(\ell_0(x,\xi )-\lambda )^{-1}$, as noted. Now
\begin{align*}
%\aligned
(\ell-\lambda )\# \tilde \ell_{0,\lambda }&=(\ell_0-\lambda
)\tilde\ell_{0,\lambda }+\ell_1 \tilde\ell_{0,\lambda }+
\sum_{j=1}^n\tfrac
1i\partial_{\xi _j}\ell_0\,\partial_{x_j}\tilde\ell_{0,\lambda }+l.o.t.\\
&=1+i\pmb b(x)\cdot \xi \,\tilde\ell_{0,\lambda }+
\sum_{j=1}^ni\partial_{\xi _j}\ell_0\,\partial_{x_j}\ell_0\,\tilde\ell_{0,\lambda }^2+l.o.t.\\
&=1+r, \text{ where }\\
r&=i\pmb b(x)\cdot \xi \,\tilde\ell_{0,\lambda }+
\sum_{j=1}^ni\partial_{\xi _j}\ell_0\,\partial_{x_j}\ell_0\,\tilde\ell_{0,\lambda }^2+l.o.t.
\end{align*}
Since $(1+r)\#(1-r)=1-r\# r$ with $r\# r$ of order $-2$, it follows that 
\begin{equation*}
(\ell-\lambda )\# \tilde \ell_{0,\lambda }\# (1-r)=1+l.o.t.,
\end{equation*}
so $\ell-\lambda$ has a right parametrix
\begin{equation}
 \tilde \ell_{\lambda }= \tilde \ell_{0,\lambda }\#
 (1-r)+l.o.t.=\tilde\ell_{0,\lambda }
-i\pmb b(x)\cdot \xi\, \tilde\ell_{0,\lambda }^2-\sum_{j=1}^ni\partial_{\xi _j}\ell_0\,\partial_{x_j}\ell_0\,\tilde\ell_{0,\lambda }^3+l.o.t.\label{eq:3.3}
\end{equation}
One finds similarly a left parametrix, and concludes (by a standard
argument in elliptic theory) that $\tilde \ell_{\lambda }$ is a
two-sided parametrix.

Now the fractional powers are constructed by use of Cauchy integral formulas:

We can describe $L^a$ approximately as 
\begin{equation}L^a=\tfrac i{2\pi }\int_{\mathcal C}\lambda ^a(L-\lambda )^{-1}\, d\lambda
  ,\label{eq:3.4}
\end{equation}
with some interpretation: The curve $\mathcal C$ is chosen %in ${\mathbb
%C}\setminus \crm$ 
to encircle the
spectrum of $L$ in the positive direction, except possibly for a
finite set of eigenvalues of finite multiplicity (it can for example
consist of the rays $\{\lambda =re^{i(\frac{\pi}2 -\delta )}\mid \infty >r\ge
r_0\}$  and $\{\lambda =re^{i(\frac{3\pi }2 +\delta )}\mid r_0\le r< \infty \}$
connected by a small curve going clockwise around zero  $\{\lambda
=r_0e^{i\theta }\mid \frac{\pi }2-\delta
<\theta < \frac{3\pi }2+\delta \}$).
The integral converges
when $a<0$; for positive $a$ one can involve recomposition with
integer powers of $L$. %[Leave out further details for the moment.]

The symbol $p$
 of $P=L^a$ then satisfies
\begin{equation}
p(x,\xi )=\tfrac i{2\pi }\int_{\mathcal C}\lambda ^a\bigl((\ell_0-\lambda
)^{-1}-i \pmb b\cdot \xi (\ell_0-\lambda )^{-2}+\sum_{j=1}^ni\partial_{\xi _j}\ell_0\,\partial_{x_j}\ell_0(\ell_{0 }-\lambda )^3\bigr)\, d\lambda +l.o.t.;\label{eq:3.5}
\end{equation}
here the formula holds as it stands when $a<0$, and since the integration
curve can for each $(x,\xi )$ be replaced by a closed
curve $\mathcal C_0$ around $\ell_0(x,\xi )$, the formula generalizes to all $a$.

The first term gives, by Cauchy's formula, that the principal symbol of $P$ is 
$%$\begin{equation*}
p_0=\ell_0^a.
$ %$\end{equation*}
The next terms give
\begin{align*}
p_1&=-i\pmb b\cdot \xi\,\tfrac i{2\pi }\int_{\mathcal C_0}\lambda ^a
(\ell_0-\lambda )^{-2}\, d\lambda -\sum_{j=1}^ni\partial_{\xi
          _j}\ell_0\,\partial_{x_j}\ell_0\tfrac i{2\pi }\int_{\mathcal
          C_0}\lambda ^a(\ell_{0 }-\lambda )^3\, d\lambda \\
 &=-i\pmb b\cdot \xi\,\tfrac i{2\pi
}\int_{\mathcal C_0}\lambda ^a\tfrac d{d\lambda } (\ell_0-\lambda )^{-1}\,
   d\lambda\\
  &\qquad-{\sum}_{j=1}^ni\partial_{\xi
_j}\ell_0\,\partial_{x_j}\ell_0\tfrac i{2\pi }\int_{\mathcal C_0}\lambda
^a\tfrac12\tfrac{ d^2}{d\lambda ^2 }(\ell_{0 }-\lambda )^{-1}\, d\lambda\\
&=i\pmb b\cdot \xi\,\tfrac i{2\pi
}\int_{\mathcal C_0}\bigl(\tfrac d{d\lambda }\lambda ^a \bigr)(\ell_0-\lambda )^{-1}\,
                                                                            d\lambda\\
  &\qquad-{\sum}_{j=1}^ni\partial_{\xi
_j}\ell_0\,\partial_{x_j}\ell_0\tfrac i{2\pi }\int_{\mathcal C_0}\tfrac{ d^2}{d\lambda ^2 }\lambda
^a\tfrac12(\ell_{0 }-\lambda )^{-1}\, d\lambda\\
&=i\pmb b\cdot \xi\,\tfrac i{2\pi
}\int_{\mathcal C_0}a\lambda ^{a-1} (\ell_0-\lambda )^{-1}\,d\lambda \\
  &\qquad-{\sum}_{j=1}^ni\partial_{\xi
_j}\ell_0\,\partial_{x_j}\ell_0\tfrac i{2\pi }\int_{\mathcal C_0}\tfrac12 a(a-1)\lambda
^{a-2}(\ell_{0 }-\lambda )^{-1}\, d\lambda\\
&=i\pmb b\cdot\xi \, a\ell_0^{a-1}-\tbinom a 2{\sum}_{j=1}^ni\partial_{\xi
_j}\ell_0\,\partial_{x_j}\ell_0\, \ell_0^{a-2},
\end{align*}
evaluated for $\xi \ne 0$ so that the negative powers of $\ell_0$ make sense.
Thus we find that the symbol of $P$ satisfies (for $\xi \ne 0$):
\begin{align} \label{eq:3.6}
p(x,\xi )&=p_0(x,\xi )+p_1(x,\xi )+l.o.t.,\\ \nonumber
p_0&=
\ell_0^a, \quad p_1=\ell_0^a(ia\pmb b\cdot \xi \ell_0^{-1}-\tbinom a 2{\sum}_{j=1}^ni\partial_{\xi
_j}\ell_0\,\partial_{x_j}\ell_0\,\ell_0^{-2})).
\end{align}

\section {The boundary term $B$ in Green's formula}\label{sec4}

We know from \cite{G18} Th.\ 4.1 that the symbol of the $\psi $do $B$ entering
in Green's formula for $P$ in the half-space situation equals  (for
$|\xi '|
\ge 1$):
\begin{equation}
 b(x',\xi ')=\lim_{z_n\to
0+}\check q(x',0,\xi ',z_n)-\lim_{z_n\to
0-}\check q(x',0,\xi ',z_n)\equiv \operatorname{jump}\mathcal F^{-1}_{\xi _n\to z_n}q(x',0,\xi ),\label{eq:4.1}
\end{equation}
where $q(x,\xi )$ is the symbol of $Q=\Xi_-^{-a}P\Xi_+^{-a}$, and
$\check q(x,\xi ',z_n)=\mathcal F^{-1}_{\xi _n\to z_n}q$. For $P=L^a$ as
above we shall first describe $q$ with two precise terms.

Recalling that  $\Xi _\pm^{-a}$ are the (generalized) $\psi $do's with
symbols $\chi _\pm^{-a}=(\ang{\xi '}\pm i\xi _n)^{-a}$, we have that
the symbol of $Q$ satisfies, by \eqref{eq:3.6} and the Leibniz
product formula \eqref{eq:2.3},
\begin{align}\label{eq:4.2}
q(x,\xi )&=\chi _-^{-a}\# p\# \chi _+^{-a}=\chi _-^{-a}\#
p \chi _+^{-a}\\ \nonumber
&=\ang{\xi } ^{-2a}(p_0+p_1 )+{\sum}_{j\le
n}\tfrac1i \partial_{\xi _j}\chi _-^{-a}\,\partial_{x_j}p_0\,\chi
_+^{-a}+l.o.t.,
\end{align}
using that $\chi _{-}^{-a}\chi _+^{-a}=\ang\xi ^{-2a}$. Here
\begin{equation}
\ang{\xi } ^{-2a}(p_0+p_1 )
=\bigl(\tfrac{l_0}{\ang{\xi } ^2}\bigr)^a(1+ia\pmb b\cdot \xi
\ell_0^{-1}
-\tbinom a 2{\sum}_{j=1}^ni\partial_{\xi
  _j}\ell_0\,\partial_{x_j}\ell_0\,\ell_0^{-2}),\label{eq:4.3}
\end{equation}
and, since $\partial_{\xi _j}\ang {\xi '}=\xi _j\ang {\xi '}^{-1}$,
\begin{align}\nonumber
{\sum}_{j\le
n}\tfrac1i \partial_{\xi _j}\chi _-^{-a}\,\partial_{x_j}p_0\,\chi
_+^{-a}&={\sum}_{j<n}i a\chi _-^{-a-1}\partial_{\xi _j}\ang{\xi '}\,
a\ell_0^{a-1}\partial_{x_i}\ell_0\,\chi _+^{-a}\\ \label{eq:4.4}
&\quad +i a\chi _-^{-a-1}(-i)
a\ell_0^{a-1}\partial_{x_n}\ell_0\,\chi _+^{-a}\\ \nonumber
&=ia^2\bigl(\tfrac{l_0}{\ang{\xi } ^2}\bigr)^a\chi _-^{-1}\ell_0^{-1}({\sum}_{j<n}\xi _j\ang{\xi '}^{-1}\partial_{x_j}\ell_0-i\partial_{x_n}\ell_0).
\end{align}
We observe that the resulting expressions  have the form of a product
of $\bigl(\frac{l_0}{\ang{\xi } ^2}\bigr)^a$ with a rational function
of $\xi _n$ that is $O(\xi _n^{-1})$ for $|\xi _n|\to \infty $. This
prepares the way for evaluating $b(x',\xi ')$ in \eqref{eq:4.1}, but we first have
to deal also with the factor  $\bigl(\frac{l_0}{\ang{\xi } ^2}\bigr)^a$. Write
\begin{equation}
\frac{l_0}{\ang\xi ^2}=\frac{a_{nn}\ang{\xi }^2+{\sum}'_{jk}a_{jk}\xi _j\xi _k-a_{nn}\ang{\xi '}^2}{\ang\xi ^2}
=a_{nn}(1+\frac{c(x,\xi )}{\ang{\xi }^2}),\label{eq:4.5}
\end{equation}
where $\sum'$ denotes the sum omitting the term with $j=k=n$, and\begin{equation}
c(x,\xi
)=a_{nn}^{-1}\bigl({\sum}_{j<n}(a_{jn}+a_{nj})\xi _j\xi _n +{\sum}_{j,k<n}a_{jk}
\xi _j\xi _k\bigr)-\ang{\xi '}^2.\label{eq:4.6}
\end{equation} 
Note that
$c(x,\xi )$ is a second-order polynomial in $\xi $ of order 1 in
$\xi _n$. Then, by Taylor expansion of $(1+t)^a$,
\begin{equation}
\bigl(\frac{l_0}{\ang{\xi } ^2}\bigr)^a=a_{nn}^a\bigl(1+\frac{c(x,\xi
)}{\ang{\xi }^2}\bigr)^a=a_{nn}^a\bigl(1-a\frac{c(x,\xi
)}{\ang{\xi }^2}+\binom a 2 \frac{c(x,\xi
)^2}{\ang{\xi }^4}+O(\xi _n^{-3})
\bigr),\label{eq:4.7}
\end{equation}
for $|\xi _n|\to\infty $. (Only the expansion up to $O(\xi
_n^{-2})$ is used in the following.)

This leads to:

\begin{thm}\label{Theorem 4.1}
 Let $P=L^a$ on ${\mathbb R}^n$, where $L$ is a second-order  strongly elliptic
differential operator {\rm \eqref{eq:3.1}}, and let $q$ be the symbol of $Q=\Xi
_-^{-a}P\Xi _+^{-a}$, defined relative to the halfspace $\rnp$. Let $|\xi '|\ge 1$. As a function of $\xi _n$,
$q$ satisfies:
\begin{align} \nonumber
a_{nn}^{-a}q(x,\xi )&=1-\frac {aa_{nn}^{-1}{{\sum}}_{j<n}(a_{jn}+a_{nj})\xi _j\xi _n}{\ang\xi ^2}+ia
b_n\xi _n\ell_0^{-1}-\tbinom a2 i\xi _n\partial_{x_n}a_{nn}\,\ell_0^{-1}\\
&\quad +a^2a_{nn}^{-1}\chi _-^{-1}({\sum}_{j<n}i\xi
_j\ang{\xi '}^{-1}\partial_{x_j}a_{nn}+\partial_{x_n}a_{nn})
+O(\xi _n^{-2}).
 \label{eq:4.8}
\end{align}
\end{thm}

\begin{proof} In the various expressions we absorb terms that are
  $O(\xi _n^{-2})$ in the remainder. For \eqref{eq:4.7}
  we have, using \eqref{eq:4.6}:
\begin{equation}
\bigl(\frac{l_0}{\ang{\xi } ^2}\bigr)^a=a_{nn}^a\bigl(1-\frac{aa_{nn}^{-1}\xi _n{{\sum}}_{j<n}(a_{jn}+a_{nj})\xi _j}{\ang{\xi }^2}+O(\xi _n^{-2})
\bigr).\label{eq:4.9}
\end{equation}
Now consider the terms in \eqref{eq:4.3}. For the first term, we  note:
\begin{equation}
ia\pmb b\cdot \xi \ell_0^{-1}=ia b_n\xi _n\ell_0^{-1}+O(\xi
_n^{-2}).\label{eq:4.10}
\end{equation}
For the second term we shall use that $\ell_0=a_{nn}\xi _n^2+O(\xi
_n)$ implies
\begin{equation}
\xi _n^2\ell_0^{-1}=(a_{nn}^{-1}\ell_0+O(\xi _n))\ell_0^{-1}=a_{nn}^{-1}+O(\xi _n^{-1}),\label{eq:4.11}
\end{equation}
in the calculation
\begin{align} \nonumber
-\tbinom a 2{\sum}_{j=1}^ni\partial_{\xi
_j}\ell_0\,\partial_{x_j}\ell_0\,\ell_0^{-2}&=-\tbinom a2ia_{nn}\xi
_n\partial_{x_n}\ell_0\,\ell_0^{-2}+O(\xi _n^{-2})\\ \label{eq:4.12}
&=-\tbinom a2i\xi _n^3a_{nn}\partial_{x_n}a_{nn}\,\ell_0^{-2}+O(\xi
_n^{-2})\\ \nonumber
&=-\tbinom a2 i\xi _n\partial_{x_n}a_{nn}\,\ell_0^{-1}+O(\xi
_n^{-2}).\end{align}
The expression in \eqref{eq:4.4} satisfies
\begin{align}
\label{eq:4.13}
ia^2\chi _-^{-1}\ell_0^{-1}&({\sum}_{j<n}{\xi _j}\ang{\xi
'}^{-1}\partial_{x_j}\ell_0-i\partial_{x_n}\ell_0)\\ \nonumber
&=ia^2\xi _n^2\chi
_-^{-1}\ell_0^{-1}({\sum}_{j<n}\xi _j\ang{\xi '}^{-1}\partial_{x_j}a_{nn}-i\partial_{x_n}a_{nn})+O(\xi _n^{-2})\\ \nonumber
&=a^2a_{nn}^{-1}\chi
_-^{-1}({\sum}_{j<n}i\xi _j\ang{\xi
'}^{-1}\partial_{x_j}a_{nn}+\partial_{x_n}a_{nn})+O(\xi _n^{-2}),
\end{align}
again using \eqref{eq:4.11}.
This gives \eqref{eq:4.8}, when the terms are collected in \eqref{eq:4.2}.\qed
\end{proof}

To find $b(x',\xi ')$, the jump of $\check q$ at $z_n=0$, we appeal to
a little of the knowledge used in the Boutet de Monvel calculus.
Recall from \cite{G96,G09} that the space $\mathcal H^+=\mathcal
F(e^+r^+\mathcal S({\mathbb R}))$ consists of functions of $\xi _n\in{\mathbb R}$ that
are $O(\xi _n^{-1})$ at infinity and extend holomorphically into the
lower halfplane 
${\mathbb C}_-$ (with further estimates), and that there is a similar
space  $\mathcal H^-_{-1}=\mathcal
F(e^-r^-\mathcal S({\mathbb R}))$ consisting of functions that
 extend holomorphically into the upper halfplane 
${\mathbb C}_+$; it is the conjugate
space of $\mathcal H^+$. All we shall use here is that the fractional
terms 
in $q$ can (for $|\xi '|\ge 1$) be decomposed 
into parts in
$\mathcal H^+$ and $\mathcal H^-_{-1}$ with respect to $\xi _n$, in view of the formulas
\begin{align} \nonumber
\ang\xi ^{-2}&=\frac 1{\ang{\xi '}^2+\xi _n^2}=\frac 1{2\ang{\xi
'}}\bigl(\frac 1{\ang{\xi '}+i\xi _n}+\frac 1{\ang{\xi '}-i\xi
_n}\bigr),\\ \nonumber
 \frac{-2i\xi
_n}{\ang\xi ^2}&=
\frac 1{\ang{\xi '}+i\xi _n}-\frac 1{\ang{\xi '}-i\xi
_n},\\ \nonumber
 \chi _-^{-1}&=\frac 1 {\ang{\xi '}-i\xi _n},
 \\ \label{eq:4.14}
\ell_0 ^{-1}&=\frac 1{{\sum} a_{jk}\xi _j\xi _k}=\frac1{a_{nn}(\sigma _++i\xi
_n)(\sigma _--i\xi _n)}\\ \nonumber
&=\frac 1{a_{nn}(\sigma _++\sigma _-)}\bigl(\frac
1{\sigma _++i\xi _n}+\frac 1{\sigma _--i\xi _n}\bigr),\\ \nonumber
-i\xi _n\ell_0 ^{-1}&=\frac {-i\xi _n}{a_{nn}(\sigma _++\sigma _-)}\bigl(\frac
1{\sigma _++i\xi _n}+\frac 1{\sigma _- -i\xi _n}\bigr)\\ \nonumber
&=\frac {1}{a_{nn}(\sigma _++\sigma _-)}\bigl(\frac
{\sigma _+}{\sigma _++i\xi _n}-\frac {\sigma _-}{\sigma _--i\xi _n}\bigr).
\end{align}
Here $\pm i\sigma _{\pm}$ are the roots of ${\sum}_{jk}a_{jk}\xi _j\xi _k$  in ${\mathbb C}_\pm$ with respect to $\xi _n$,
 respectively
(then $\text{Re }\sigma _\pm>0$).
When $\operatorname{Re}\sigma >0$, $(\sigma -i\xi
_n)^{-1}\in \mathcal H^-_{-1}$ and $(\sigma +i\xi
_n)^{-1}\in \mathcal H^+$, and 
(with $H$ equal to the Heaviside function $1_{\crp}$)
\begin{equation}
\mathcal F_{\xi _n\to z_n}^{-1}\frac 1{\sigma +i\xi _n}=H(z_n)e^{-\sigma
z_n}, \quad \mathcal F_{\xi _n\to z_n}^{-1}\frac 1{\sigma -i\xi _n}=H(-z_n)e^{\sigma
z_n};\label{eq:4.15}
\end{equation}
these functions have the limit 1 for $z_n\to 0+$,
resp.\ $z_n\to 0-$.

Then from \eqref{eq:4.14} follows  for example: 
\begin{align} \nonumber
\text{(i) }&\operatorname{jump}\mathcal F_{\xi _n\to z_n}^{-1}\ang\xi ^{-2}=0,\\
\text{(ii) }&\operatorname{jump}\mathcal F_{\xi _n\to z_n}^{-1}\xi _n\ang\xi
^{-2}=i,\label{eq:4.16}\\ \nonumber
\text{(iii) }&\operatorname{jump}\mathcal F_{\xi _n\to z_n}^{-1}\chi
               _-^{-1}=-1,\\ \nonumber
\text{(iv) }&\operatorname{jump}\mathcal F_{\xi _n\to z_n}^{-1}\xi
_n\ell_0^{-1}=ia_{nn}^{-1}.
              \end{align}

This leads to:

\begin{thm}\label{Theorem 4.2}
 Assumptions as in Theorem \ref{Theorem 4.1}. The symbol
$b(x',\xi ')$ defined by {\rm \eqref{eq:4.1}} satisfies:
\begin{align} \nonumber
b(x',\xi ')&=-aa_{nn}^{a-1}{{\sum}}_{j<n}(a_{jn}+a_{nj})i\xi
_j-aa_{nn}^{a-1}b_n+\tbinom a2 a_{nn}^{a-1}\partial_{x_n}a_{nn}\\
&\quad-a^2a_{nn}^{a-1}({\sum}_{j<n}i\xi
_j\ang{\xi '}^{-1}\partial_{x_j}a_{nn}+\partial_{x_n}a_{nn}),\label{eq:4.17}
\end{align}
all coefficients evaluated at $x_n=0$.
\end{thm}

\begin{proof} Consider $q(x',0,\xi ',\xi _n)$ described by \eqref{eq:4.8}
multiplied by $a_{nn}(x',0)^a$. To evaluate the inverse Fourier transform from $\xi _n$ to
$z_n$, we begin by noting that the first term contributes with
$a_{nn}(x',0)^a\delta (z_n)$, supported in $\{z_n=0\}$, which
disappears when the limits in \eqref{eq:4.1} are calculated. Moreover we will
use
that,
as already noted, symbols that are $O(\xi _n^{-2}) $ at infinity
transform to continuous functions of $z_n$, hence have jump 0. 

Now
consider the second term in the right-hand side of \eqref{eq:4.8}. Here we
find by use of \eqref{eq:4.16}(ii) that the jump, it contributes, equals
\begin{equation*}
-aa_{nn}^{a-1}{{\sum}}_{j<n}(a_{jn}+a_{nj})i\xi _j.
\end{equation*}
The third term is found by use of \eqref{eq:4.16}(iv) to contribute with
\begin{equation*}
-aa_{nn}^{a-1}b_n.
\end{equation*}
The fourth term
gives in view of \eqref{eq:4.16}(iv) the contribution 
\begin{equation*}
\tbinom a2 a_{nn}^{a-1}\partial_{x_n}a_{nn}.
\end{equation*}
The fifth term
gives by use of \eqref{eq:4.16}(iii) the contribution
\begin{equation*}
-a^2a_{nn}^{a-1}({\sum}_{j<n}i\xi
_j\ang{\xi '}^{-1}\partial_{x_j}a_{nn}+\partial_{x_n}a_{nn}).
\end{equation*}
The contributions are collected in \eqref{eq:4.17}.

\end{proof}

\begin{rem}\label{Remark 4.3} Observe that the only possibly nonlocal
contributions to $B=\Op(b(x',\xi '))$ come from the terms with $\ang{\xi
'}^{-1}\partial_{x_j}a_{nn}(x',0)$, $j<n$. So if the
first tangential
derivatives of $a_{nn}$ vanish on the boundary, $B$ is local, and otherwise it can be
nonlocal.
\end{rem}

A special case is where $L$ stems from the Laplacian.
In the reduction of the Laplacian to local coordinates described in
the Appendix, we arrive at an operator of the form (cf.\ \eqref{eq:6.14})
\begin{equation}
L=-\underline\Delta =-\underline\Delta '(y',y_n,\partial_{y'})-g(y')\partial_{y_n}-\partial_{y_n}^2.\label{eq:4.18}
\end{equation}
In comparison with the general expression \eqref{eq:3.1}, we here have
\begin{equation}
a_{nj}(y)=a_{jn}(y)\equiv 0\text{ for }j<n,\quad a_{nn}(y)\equiv 1,\quad b_n(y)=-g(y'),\label{eq:4.19}
\end{equation}
since 
$\underline\Delta '$ differentiates in the $y'$-variables only. The
derivatives of the functions $a_{nj}, a_{jn}, a_{nn}$ are zero. Hence
\eqref{eq:4.8} gives a much simplified expression for  $q$. We find, as special
cases of Theorems \ref{Theorem 4.1} and \ref{Theorem 4.2}:

\begin{cor}\label{Corollary 4.4} When $P=L^a$ with $L=-\underline\Delta $ in {\rm \eqref{eq:4.18}}, as obtained by reduction
of the Laplacian to local coordinates in the Appendix, then
\begin{equation}
q(y,\xi )=1-iag(y')\xi _n\ell_0^{-1}+O(\xi _n^{-2}),\label{eq:4.20}
\end{equation}
where $\ell_0=\ell'_0(y,\xi ')+\xi _n^2$, $\ell'_0$ being the
 principal symbol of $-\underline\Delta '$.
The symbol $b(y',\xi ')$ of $B$ is
\begin{equation}
b(y',\xi ')=ag(y').\label{eq:4.21}
\end{equation}
\end{cor}

As a slightly more general case, let $P=(-\Delta +\pmb
c(x)\cdot\nabla+c_0(x))^a$ on ${\mathbb R}^n$. On
$\partial\Omega $, 
$-\Delta +\pmb
c(x)\cdot\nabla+c_0(x)$ may be written in the form, cf.\ the Appendix:
\begin{equation}
-\Delta +\pmb
c(x)\cdot\nabla+c_0(x)=-\Delta _S-\mathcal G \frac {\partial }{\partial n}-\frac
{\partial^2 }{\partial n^2} +c_\nu \frac{\partial}{\partial
  n}+T+c_0,\label{eq:4.22} 
\end{equation}
where $\mathcal G=\operatorname{div} \nu  $, $c_\nu (x)=\pmb c(x)\cdot \nu (x)$ and $T$ is a first-order differential operator acting along
$\partial\Omega $.
In fact this decomposition extends to a tubular neighborhood of each
coordinate patch for $\partial\Omega $, as described in the Appendix for
$-\Delta $. When $c_\nu (x)$ on $\partial\Omega $ carries over to
$\underline c_\nu (y')$ at $t=0$, we extend it as constant in $t$ on the neighborhood.
Then  $P$ has the form $L_1^a$ in the
local coordinates, where 
\begin{equation}
L_1=-\underline\Delta '(y',y_n,\partial_{y'})
-\partial_{y_n}^2-(g(y')-{\underline c}_{\nu }(y'))\partial_{y_n}+\underline T+\underline c_0.\label{eq:4.23}
\end{equation}
Here we find:

\begin{cor}\label{Corollary 4.5} When $P=L_1^a$ with $L_1$ as in {\rm \eqref{eq:4.23}},
obtained by reduction to local coordinates
of the perturbed Laplacian  $-\Delta +\pmb
c(x)\cdot\nabla+c_0(x)$ (decomposed on $\partial\Omega $ as in {\rm
\eqref{eq:4.22}}),
then the corresponding symbol $q$ satisfies
\begin{equation}
q(y,\xi )=1-ia(g(y')-{\underline c}_{\nu }(y'))\xi _n\ell_0^{-1}+O(\xi _n^{-2}),\label{eq:4.24}
\end{equation}
with $\ell_0$ as in Corollary {\rm 4.4}.
The symbol $b(y',\xi ')$ of $B$ is a function of $y'$,
\begin{equation} 
b(y',\xi ')=a(g(y')-{\underline c}_{\nu }(y')).\label{eq:4.25}
\end{equation}
\end{cor}

\section {Green's formula for the fractional Laplacian and its perturbations}\label{sec5}

The above considerations in local coordinates will now be applied to
find Green's formula in the curved situation  for the 
powers of the perturbed Laplacian, in particular
for the fractional Laplacian $(-\Delta )^a$ itself.

\begin{thm}\label{Theorem 5.1} Let $\Omega $ be a smooth bounded subset of
${\mathbb R}^n$, and let  
$P=(-\Delta +\pmb
c\cdot\nabla+c_0)^a$, $a>0$. Let
$u,v\in H^{(a-1)(s)}(\comega)$, $s>a+\frac12$. When $s\ge 2a$,
\begin{align} \label{eq:5.1}
&\int_{\Omega }\bigl(P u\bar v-u\overline {P^* v}\bigr)\,
dx\\ \nonumber
&=(\gamma _1^{a-1}u -ac_\nu \gamma _0^{a-1}u,\gamma
_0^{a-1}v)_{L_2(\partial\Omega )}
-( \gamma _0^{a-1}u,\gamma _1^{a-1}v)_{L_2(\partial\Omega )}
 \\ \nonumber
&\equiv \Gamma (a)\Gamma (a+1)\int_{\partial\Omega }\bigl((\gamma
_1(\tfrac {u_k}{t^{a-1}}) -c_\nu \gamma _0(\tfrac{u_k}{t^{a-1}}))\gamma _0(\tfrac{\bar v_l}{t^{a-1}})
 -\gamma _0(\tfrac{u_k}{t^{a-1}})\gamma
_1(\tfrac{\bar v_{l}}{t^{a-1}})
\bigr)
\, d\sigma, 
\end{align}
where $c_\nu(x) =\pmb c(x)\cdot \nu (x)$.
The formula extends to general $s$ with $s>a+\frac12$, when the left-hand side is
replaced by 
\begin{equation}
\ang{r^+Pu,v}_{\ol H^{-a+\frac12+\varepsilon }(\Omega ), \dot
H^{a-\frac12-\varepsilon }(\comega)}-\ang{u,r^+P^*v}_{\dot H^{a-\frac12-\varepsilon }(\comega),\ol
H^{-a+\frac12+\varepsilon }(\Omega )}.\label{eq:5.2}
\end{equation}
In particular, the fractional Laplacian $(-\Delta )^a$ satisfies
\begin{align} \nonumber
\int_{\Omega }\bigl(&(-\Delta )^au\bar v-u\overline {(-\Delta )^a v}\bigr)\,
dx=(\gamma _1^{a-1}u,\gamma
_0^{a-1}v)_{L_2(\partial\Omega )}
-( \gamma _0^{a-1}u,\gamma _1^{a-1}v)_{L_2(\partial\Omega )}
\\
&\equiv \Gamma (a)\Gamma (a+1)\int_{\partial\Omega }\bigl(\gamma
_1(\tfrac u{t^{a-1}})\gamma _0(\tfrac{\bar v}{t^{a-1}})
 -\gamma _0(\tfrac{u}{t^{a-1}})\gamma
_1(\tfrac{\bar v}{t^{a-1}})\bigr)
\, d\sigma . \label{eq:5.3}
\end{align}
\end{thm}

\begin{proof} 
It is shown in \cite{G18} Th.\ 4.4 for operators satisfying the
$a$-transmission condition how the formula for a
general domain $\Omega $ is deduced from the knowledge of Green's
formula in flat cases $\Omega =\rnp$, by use of
local coordinates. We shall follow that construction for our special
operators, and rather than taking up space by repeating the whole
proof, we shall just
explain the needed ingredients. 

The general transformation rule is
\eqref{eq:6.17}. We first note that when $L$ corresponds to $\underline L$,
then in view of Seeley's analysis of $a$'th powers of $\psi $do's by
passage via the resolvent and a Cauchy integral formula (recalled in
Section 3), the terms in
the symbol of  $\underline P=\underline {(L^a)}$, carried over from $P=L^a$
by the coordinate change,
are consistent with the terms in the symbol of $(\,\underline L\,)^a$.  
This will be used with $L= -\Delta +\pmb
c(x)\cdot\nabla+c_0(x)$, reduced to the form \eqref{eq:4.23} in a local
coordinate system.

The set
$\comega$ is covered by a system of bounded open sets $U_0,\dots, U_{I_1}$, with
diffeomorphisms $\kappa _i\colon U_i\to V_i\subset \rn$ such that
$U_i^+=U_i\cap\Omega $ is mapped to $V_i^+=V_i\cap\rnp$, and
$U'_i=U_i\cap \partial\Omega $ is mapped to $V'_i=V_i\cap {\mathbb
R}^{n-1}$ (${\mathbb R}^{n-1}=\partial\crnp$), the  restriction of
$\kappa _i$ to $U'_i$
denoted $\kappa '_i$.
The diffeomorphism is chosen such that the interior normal $\nu (x')$ at $x'\in\partial\Omega $ 
defines a normal coordinate near $\partial\Omega $: 
\begin{equation*}
\kappa _i \text{ maps }x'+t\nu (x')\in U_i \text{ to }(y',t)\in
V_i\subset {\mathbb R}^{n-1}\times {\mathbb R} 
\end{equation*}
near $\partial\Omega $ (with $y'= \kappa '_i(x')$).  We shall denote the inverses 
$\kappa _i^{-1}=\tilde \kappa _i$, $(\kappa '_i)^{-1}=\tilde \kappa
'_i$.

There is a partition of unity 
$\varrho _k$, $ k=0,\dots,J_0$, with ${\sum}_k\varrho _k=1$ on a
neighborhood of $\comega $, \emph{subordinate} to the covering, in the
sense that for any two functions $\varrho _k,\varrho _l$ there is an
$i=i(k,l)$ in $\{0,\dots,I_1\}$ such that $\supp \varrho _k\cup \supp
\varrho _l\subset U_{i(k,l)}$. Moreover, nonnegative functions $\psi
_k$ and $\zeta _k\in C_0^\infty (U_i)$ are chosen such that $\zeta
_k\varrho _k=\varrho _k$ and $\psi _k\zeta _k=\zeta _k$.

Now a given $u\in H^{(a-1)(s)}(\comega)$ can be decomposed in this
space as a sum
$u={\sum}_{k\le J_0}u_k+r$, where $\supp u_k\subset \supp \zeta _k\subset
U_i$, and $r\in \dot C^\infty (\comega)$ does not contribute to the
boundary integrals. There is a similar decomposition $v={\sum}_{l\le
J_0}v_l+r'$ of $v\in H^{(a-1)(s)}(\comega)$. The operators $P$ and
$P^*$ can in their action on $u_k$ and $v_l$ in the scalar products be
replaced by 
\begin{equation*}
P_{kl}=\psi _lP\psi _k,\quad P^*_{kl}= \psi _kP^* \psi _l.
\end{equation*}
As earlier noted, the action of the operators in local coordinates follows the rule
recalled in \eqref{eq:6.17}; we indicate localized operators and functions
by underlines. It is shown in Th.\ 4.4 of \cite{G18} how the
contribution from $u_k,v_l$ is reformulated and worked out as
\begin{align} \nonumber
(P_{kl}u_k,v_l)_{\Omega ^+_i}& - (u_k,P^*_{kl}v_l)_{\Omega ^+_i} 
 =(r^+\underline P_{kl}\underline
u_k,J\underline v_l)_{\rnp}-(\underline
u_k,r^+(\underline P_{kl})^{(*)}J\underline v_l)_{\rnp}\\ \label{eq:5.4}
 =(\underline s_{kl,0}&\gamma _1^{a-1}\underline u_k,\gamma
_0^{a-1}J\underline  v_l)_{{\mathbb R}^{n-1}}-(\underline s_{kl,0}\gamma _0^{a-1}\underline
                        u_k,\gamma _1^{a-1}J\underline v_l)_{{\mathbb R}^{n-1}}\\ \nonumber
 & +(\underline B_{kl}\gamma
_0^{a-1}\underline u_k,\gamma _0^{a-1}J\underline v_{l})_{{\mathbb R}^{n-1}},
\end{align}
where $J=|\partial
x/\partial(y',t)|$, the absolute value of the functional determinant
of $\tilde\kappa $ going
from the local coordinates $(y',t)$ to the
given coordinates $x$. (We omit marking the operators with 
$(i)$ as in \cite{G18} indicating the dependence on the coordinate patch.)

The effect of $J$ in the boundary values with respect to $\underline
v_l$ is as follows:
\begin{align} \nonumber
  \gamma _0^{a-1}(J\underline v_l)&=\Gamma (a)\gamma _0\bigl(
                                    \frac{J\underline v_l}{t^{a-1}}\bigr)
                                    =\gamma _0(J)\gamma _0 ^{a-1}(\underline
                                  v_l),\\ \label{eq:5.5} 
\gamma _1^{a-1}(J\underline v_l)&=\Gamma (a+1)\gamma
_0\bigl(\frac \partial{\partial t}\frac{J\underline
v_l}{t^{a-1}}\bigr)\\ \nonumber
&=\Gamma (a+1)\gamma
_0(J)\gamma _1\bigl(\frac{\underline
v_l}{t^{a-1}}\bigr)+a\Gamma (a)\gamma
_1(J)\gamma _0\bigl(\frac{\underline
v_l}{t^{a-1}}\bigr)\\ \nonumber
&=\gamma _0(J)\gamma _1 ^{a-1}(\underline v_l)+a\gamma
_1 (J)\gamma _0 ^{a-1}(\underline v_l).
\end{align}
Here we recall from the Appendix that $\gamma _0(J)=J_0$, the factor
entering in integration formulas over $\partial\Omega $, and \begin{equation}
\gamma
_1(J)=J_1=J_0g,\quad g=\operatorname{div}\nu .\label{eq:5.6}
\end{equation}
Hence
\begin{equation}
\gamma _0^{a-1}(J\underline v_l)=J_0\gamma _0^{a-1}(\underline v_l),\quad \gamma _1^{a-1}(J\underline v_l)=J_0\gamma _1 ^{a-1}(\underline v_l)+aJ_0g\gamma _0 ^{a-1}(\underline v_l).\label{eq:5.7}
\end{equation}

Now we apply Th.\ 4.1 of \cite{G18}, using the formula for the
localized version of $P=(-\Delta +\pmb c\cdot \nabla + c_0)^a$ described in Section 4.
 Since the cutoff
functions $\underline\psi _k,\underline\psi _l$ are 1 on $\supp
\underline u_k$ resp.\ $\supp \underline v_l$,
they can be disregarded in the formulas.

As shown in Corollary 4.5, $\underline B$ for $\underline P$
is the multiplication by $a(g(y')-\underline c_\nu (y'))$. 
Then \eqref{eq:5.4} takes the form, in view
of \eqref{eq:5.7},
\begin{align} \nonumber
&(r^+\underline P_{kl}\underline
u_k,J\underline v_l)_{\rnp}-(\underline
u_k,r^+(\underline P_{kl})^{(*)}(J\underline v_l))_{\rnp}\\ \nonumber
& =(\underline s_{kl,0}\gamma _1^{a-1}\underline u_k,J_0\gamma
_0^{a-1}\underline v_l)_{{\mathbb R}^{n-1}}-(\underline s_{kl,0}\gamma _0^{a-1}\underline
u_k,J_0\gamma _1^{a-1}\underline v_l+aJ_0g\gamma _0^{a-1}\underline
v_l)_{{\mathbb R}^{n-1}}\\ \label{eq:5.8}
&\quad +(a(g-\underline c_\nu )\gamma
_0^{a-1}\underline u_k,J_0\gamma _0^{a-1}\underline v_{l})_{{\mathbb
                                            R}^{n-1}}\\ \nonumber
& =(\underline s_{kl,0}\gamma _1^{a-1}\underline u_k,J_0\gamma
_0^{a-1}\underline v_l)_{{\mathbb R}^{n-1}}-(\underline s_{kl,0}\gamma _0^{a-1}\underline
u_k,J_0\gamma _1^{a-1}\underline v_l)_{{\mathbb R}^{n-1}}\\ \nonumber
&\quad -(a\underline c_\nu \gamma
_0^{a-1}\underline u_k,J_0\gamma _0^{a-1}\underline v_{l})_{{\mathbb R}^{n-1}},\end{align}
where the terms with $ag$ \emph{cancelled out!}
Here $p(y',0,0,1)=(\ell_0(y',0,0,1))^a=1$ (since the coefficient of
$\partial_{y_n}^2$ in $\ell_0$ is 1), so the factor $\underline s_{kl,0}$ is simply $\underline \psi
_k\underline \psi _l$, which is 1 on  $\supp\underline u_k\cap\supp \underline v_l$, and the last display in \eqref{eq:5.8} simplifies to
\begin{equation}
(\gamma _1^{a-1}\underline u_k,J_0\gamma
_0^{a-1}\underline v_l)_{{\mathbb R}^{n-1}}-(\gamma _0^{a-1}\underline
u_k,J_0\gamma _1^{a-1}\underline v_l)_{{\mathbb R}^{n-1}}-(a\underline c_\nu \gamma _0^{a-1}\underline
u_k,J_0\gamma _0^{a-1}\underline v_l)_{{\mathbb R}^{n-1}}
.\label{eq:5.9}
\end{equation}
Expressed in $x$-coordinates, this gives
\begin{align*}
&(\gamma _1^{a-1}u_k,\gamma
_0^{a-1}v_l)_{L_2(\partial\Omega )}
-( \gamma _0^{a-1}{u_k},\gamma _1^{a-1}v_l)_{L_2(\partial\Omega )}-( ac_\nu \gamma _0^{a-1}{u_k},\gamma _0^{a-1}v_l)_{L_2(\partial\Omega )}
\\
&\equiv \Gamma (a)\Gamma (a+1)\int_{\partial\Omega }\bigl(\gamma
_1(\tfrac {u_k}{t^{a-1}})\gamma _0(\tfrac{\bar v_l}{t^{a-1}})
 -\gamma _0(\tfrac{u_k}{t^{a-1}})\gamma
_1(\tfrac{\bar v_{l}}{t^{a-1}})  -c_\nu \gamma _0(\tfrac{u_k}{t^{a-1}})\gamma
_0(\tfrac{\bar v_{l}}{t^{a-1}})
\bigr)
\, d\sigma, \end{align*}
 and a summation over $k,l$ leads to \eqref{eq:5.1}. 

The
 validity on lower-order function spaces is accounted for in
 \cite{G18}, and the formula for $(-\Delta )^a$ is a special case
 where $c_\nu =0$. 

\end{proof}

\section {Appendix. Localization of the Laplacian}\label{sec6}

The basic arguments in \cite{G18} depend on a study of
pseudodifferential boundary value problems, reduced from the general
situation where $\partial\Omega $ is a hypersurface in ${\mathbb R}^n$ to
the situation where $\partial\Omega $ equals the boundary ${\mathbb
R}^{n-1}$ of the half-space $\rnp=\{x\in
{\mathbb R}^n\mid x_n> 0\}$ (where $(x_1,\dots, x_{n-1})=x'$). As a
preparation for  seeing how
such coordinate changes affect $(-\Delta )^a$, we here
investigate their effect on $-\Delta $ itself (the case $a=1$).

In the following, $S$ is a smooth hypersurface in ${\mathbb R}^n$ (e.g.\
a piece of $\partial\Omega $)
defined as
\begin{equation}
S=\chi (V'), \quad V'\text{ open }\subset {\mathbb R}^{n-1},\label{eq:6.1}
\end{equation}
where $\chi =(\chi _1(y'),\dots ,\chi _ n(y')) $ is a smooth injective
mapping. With $\nu (\chi (y'))$ denoting a unit normal to $S$
at each point $\chi (y')\in S$ (orthogonal to the tangent vectors
$\partial\chi /\partial y_j$, $j=1,\dots,{n-1}$, its orientation
depending continuously on $y'$), we parametrize a tubular neighborhood $U$ of $S$
by a diffeomorphism
\begin{equation}
\tilde\kappa \colon y=(y',t)\mapsto x=\chi (y')+t\nu (\chi (y')),\label{eq:6.2}
\end{equation}
from $V=V'\times\,]-\delta , \delta [\,$ to $U$ (possibly
after replacing $U$, $V$ and $S$ by smaller sets). The functional
matrix is 
\begin{equation}
\frac{\partial x}{\partial y}=\frac{\partial\tilde\kappa }{\partial
(y',t)}=
\begin{pmatrix} \dfrac{\partial\chi }{\partial y'}+t\dfrac{\partial\nu (\chi
)}{\partial y'}& \nu (\chi )\end{pmatrix}, \label{eq:6.3}
\end{equation}
written as an $n\times(n-1)$-block next to an $n\times 1$-block.
We can view this as $M+tN$, where 
\begin{equation}
M=\begin{pmatrix} \dfrac{\partial\chi }{\partial y'}& \nu (\chi
)\end{pmatrix}=\frac{\partial x}{\partial y}\Big|_{t=0},\quad N =\begin{pmatrix} \dfrac{\partial\nu (\chi
)}{\partial y'}& 0\end{pmatrix}, \label{eq:6.4}
\end{equation}

The Jacobian is the absolute value of the functional determinant
\begin{equation}
J=|\det\frac{\partial x}{\partial y}|=|\det(M+tN)|.\label{eq:6.5}
\end{equation}
To fix the ideas, assume that $\det M>0$, so that $J=\det(M+tN)$ for
small  $t$.

In comparison with the notation in \cite{G18}, we are leaving out the
indexation by $i$ in Remark 4.3 there, $\tilde\kappa $ is the
inverse of the diffeomorphism denoted $\kappa _i$ there, and $\chi $ is the inverse
of the mapping  $\kappa _i'$.

Denote $J|_{t=0}=J_0$, note that it equals $\det M$. It is well-known
that integration over $S$ can be described via the local coordinates as
follows:
When $\varphi (x)$ is a function on $S$, denote by $\underline \varphi
(y')$ the corresponding
function on $V'$, that is,
\begin{equation*}
\varphi (\chi (y'))=\underline\varphi (y'). %\label{eq:6.6}
\end{equation*}
The rule for integration is then
\begin{equation}
\int_{S}\varphi (x)\, d\sigma =\int_{V'}\underline\varphi (y')J_0\,
dy' \label{eq:6.7}
\end{equation}
(the appropriate ``area-element'' is $J_0\,dy'$). This is found in
introductory textbooks on differential geometry; note that $J_0$ can also be described by 
\begin{equation}
J_0=\sqrt{\det (M^tM)}\,=\sqrt{\det (\tfrac{\partial \chi }{\partial y_j}\cdot
\tfrac{\partial \chi }{\partial y_k})_{j,k\le n-1}}\,.\label{eq:6.8}
\end{equation}
%[Reference? Har kun Driver-noterne.] 

For Green's formula we shall moreover need the value of the
$t$-derivative of $J$ at $t=0$, that we calculate here for
completeness:

\begin{lem}\label{Lemma A.1} Assume that $\det M>0$. The value of
$J_1=\partial_tJ|_{t=0}$ at the points of $S$ is \begin{equation}
  J_1=J_0\operatorname{div}\nu .\label{eq:6.9}
\end{equation} 
\end{lem}

\begin{proof} Fix $x\in S$. Since $J(t)=\det (M+tN)$ is a polynomial of
degree $n$ in $t$, $J_1$ is the coefficient of the first  power $t$. Now
with $s=1/t$, \begin{align*}
\det(M+tN)&=t^n\det M \det (s+NM^{-1})\\
&=t^n\det M
(s^n+\operatorname{trace}(NM^{-1}) s^{n-1}+\dots +\det(NM^{-1})),
\end{align*}
so the coefficient of $t$ is $\det M\operatorname{trace}(NM^{-1})$.
Here
\begin{equation*}
NM^{-1}=\begin{pmatrix}  \dfrac{\partial\nu (\chi
)}{\partial y'}& 0\end{pmatrix} \frac{\partial y}{\partial x}= \frac{\partial\nu (\chi
)}{\partial y} \frac{\partial y}{\partial x}= \frac{\partial\nu (\chi
)}{\partial x} .
\end{equation*}
 The trace of this matrix equals
\begin{equation}
\operatorname{trace}(NM^{-1})=\partial_{x_1}\nu _1+\dots +\partial_{x_n}\nu _n=\operatorname{div}\nu .\label{eq:6.10}
\end{equation}
Since $\det M=J_0$, it follows that $\det M
\text{trace}(NM^{-1})=J_0\operatorname{div}\nu $.
\end{proof}

The function $\operatorname{div}\nu $ on $S$ represents the
mean curvature, modulo a dimensional factor.

It is known that $\Delta u$ on $S$ has the form
\begin{equation}
\Delta u=\Delta _Su+\mathcal G \frac {\partial u}{\partial n}+\frac
{\partial^2 u}{\partial n^2} \quad\text{ on }S,\label{eq:6.11}
\end{equation}
cf.\ e.g.\ Hsiao and Wendtland \cite{HW08} (with reference to Leis
1967) and Duduchava, Mitrea and Mitrea \cite{DMM06} (with reference to
G\"unther 1934). Here $\Delta _S$ is the Laplace-Beltrami operator on
$S$, $\partial u/\partial n$ is the normal derivative $\partial
u/\partial n={\sum}_{j=1}^n\nu _j\partial_{x_j}u$, and $\mathcal
G=\operatorname{div}\nu 
$. In our local
coordinates, $\partial /\partial n$ corresponds to $\partial/\partial
t$, and $\Delta _S$ corresponds to an operator $\underline\Delta '(y',\partial_{y'},0)$
acting with respect to $y'$ (we do not need its exact form at present).

For $|\varepsilon |<\delta $, the parallel surfaces $S_\varepsilon $
represented by \begin{equation*}
\chi _\varepsilon (y')=\chi (y')+\varepsilon \nu (\chi (y')), \quad
y'\in V',%\label{eq:6.12}
\end{equation*}
again have normals $ \nu (\chi (y'))$. Indeed, if we
denote $\nu (\chi (y'))=\tilde\nu (y')$, we have since
$\|\tilde\nu (y')\|=1$ for $y'\in V'$, that the vectors $\partial\tilde\nu /\partial
y_j$ are orthogonal to $\tilde\nu $ at the point, hence $\tilde\nu $ is orthogonal to $\partial (\chi
+\varepsilon \tilde\nu ) /\partial y_j$ at the point, for $j=1,\dots n-1$. So \eqref{eq:6.2} is
also a parametrization of a neighborhood of $S_\varepsilon $ (for $t$
near $\varepsilon $), with $\nu (\chi (y'))$ as normal at the point
$\chi (y')+\varepsilon \nu (\chi (y'))$. On $S_\varepsilon $  there is
a formula like \eqref{eq:6.11},
\begin{equation*}
\Delta u =\Delta _{S_\varepsilon }u+\mathcal G_\varepsilon \frac{\partial
u}{\partial n} +\frac{\partial ^2
u}{\partial n^2}\quad\text{ on }S_\varepsilon, %\label{eq:6.13} 
\end{equation*}
where $\mathcal G_\varepsilon $ at $\chi (y')+\varepsilon \tilde\nu
(y')$ is the same as $\mathcal G$ at $\chi (y')$, and $\partial/\partial
n$ corresponds to $\partial/\partial t$.

We conclude that in the local coordinates $(y',t)$, when
$u(x)$
corresponds to $\underline u(y)$ (i.e., $u(\tilde\kappa
(y',t))=\underline u(y',t)$, $\Delta u$ takes
the form 
\begin{equation}
\Delta u=\underline\Delta\, \underline u\equiv \Delta '(y',t,\partial_{y'})\underline u+g(y')\partial_t\underline u
+\partial_t^2\underline u,\text{ for
}(y',t)\in V,\label{eq:6.14}
\end{equation}
 where $\Delta '$ is a second-order operator differentiating only in the
 $y'$-variables, and 
\begin{equation*}
g(y')=\mathcal G(\chi (y')).
\end{equation*}

\medskip
We can now compare Green's formulas worked out in the different
coordinates.

When $S$ is a piece of the boundary $\partial\Omega $ of a smooth open
set $\Omega \subset {\mathbb R}^n$ with $\nu $ as the interior normal,
and $u$ and $v$ are supported in $U$, then, as is very well known,
\begin{equation}
(-\Delta u,v)_{U^+}-(u,-\Delta v)_{U^+}=(\gamma _1u,\gamma _0v)_{S}-(\gamma _0u,\gamma _1v)_{S},\label{eq:6.15}
\end{equation}
here $U^+=U\cap \Omega $, $\gamma _0u=u|_{\partial\Omega }$, and
$\gamma _1u=\gamma _0(\partial u/\partial n)$.

For the operator in \eqref{eq:6.14} we have for $\underline u$ and $\underline
w$ supported in $V$, denoting $V\cap \rnp=V^+$,
\begin{align} \label{eq:6.16}
(-\underline\Delta\, \underline
u,\underline w)_{V^+}&-(\underline u,-\underline\Delta
^{(*)}\underline w)_{V^+}\\ \nonumber
&=(-\Delta '\underline u-g\partial_t\underline u
-\partial_t^2\underline u,\underline w)_{V^+}-(\underline u,(-\Delta
')^{(*)}\underline w+\partial_t(g\underline w)
-\partial_t^2\underline  w)_{V^+}\\ \nonumber
&=(\gamma _1\underline u,\gamma
_0\underline w)_{V'}-(\gamma _0\underline u,\gamma
_1\underline w)_{V'}+(g\gamma _0\underline u,\gamma
_0\underline w)_{V'};
\end{align}
here the star in parentheses indicates the adjoint with respect to
$y$-coordinates, to distinguish it from the adjoint in
$x$-coordinates,
and $\gamma _1\underline w=\gamma _0(\partial \underline w/\partial
t)$ (consistently with the normal derivative).

It may seem surprising at first, that the two formulas \eqref{eq:6.15}
and \eqref{eq:6.16} are so
different, in that the latter 
has the extra term with
$g$. However, they are consistent, as we shall now show by deducing
\eqref{eq:6.15} from \eqref{eq:6.16}.

Recall, as also accounted for in \cite{G18}, that when the operator $P$ in
$x$-coordinates corresponds to $\underline P$ in $y$-coordinates:
\begin{equation}
\underline P\, \underline u=P(\underline u\circ \tilde\kappa
^{-1})\circ \tilde\kappa =(Pu)\circ \tilde\kappa ,\label{eq:6.17}
\end{equation}
with $\underline u=u\circ \tilde\kappa $, then 
\begin{equation*}
(Pu,v)_{U}=(\underline{Pu}, J\underline v)_{V},%\label{eq:6.18}
\end{equation*}
and the formal adjoint of $\underline P$ in $y$-coordinates satisfies 
\begin{equation*}
(\underline P)^{(*)}=J\underline{(P^*)}J^{-1}. %\label{eq:6.19}
\end{equation*}
Thus (for sufficiently smooth $u,v$ supported in $U$)
\begin{align*}
(-\Delta u,v)_{U^+}&-(u,-\Delta v)_{U^+}=(-\underline\Delta\, \underline
u,J\underline v)_{V^+}-(\underline u,-\underline\Delta
^{(*)}J\underline v)_{V^+}\\
&=(\gamma _1\underline u,\gamma
_0(J\underline v))_{V'}-(\gamma _0\underline u,\gamma
_1(J\underline v))_{V'}+(g\gamma _0\underline u,\gamma
_0(J\underline v))_{V'},
\end{align*}
 by \eqref{eq:6.16}.
Here 
\begin{equation}
\gamma _0(J\underline v)=J_0\gamma _0\underline v,\quad 
\gamma _1(J\underline v)=J_0\gamma _1\underline v+J_1\gamma
_0\underline v, \label{eq:6.20}
\end{equation}
where $J_0=\gamma _0J$. $J_1=\gamma _0(\partial_tJ)$ as defined above.
In view of Lemma 2.1, $J_1=J_0g$. As a result,
\begin{align*}
(-\Delta u,v)_{U^+}&-(u,-\Delta v)_{U^+}\\
&=(\gamma _1\underline u,J_0\gamma
_0\underline v)_{V'}-(\gamma _0\underline u,J_0\gamma
_1\underline v+J_0g\gamma _0\underline v)_{V'}+(g\gamma _0\underline u,J_0\gamma
_0\underline v)_{V'}\\
&=(\gamma _1\underline u,J_0\gamma
_0\underline v)_{V'}-(\gamma _0\underline u,J_0\gamma
_1\underline v)_{V'}\\
&=(\gamma _1u,\gamma
_0v)_{S}-(\gamma _0 u,\gamma
_1 v)_{S},
\end{align*}
where the terms with $g$ \emph {cancelled out}. Thus \eqref{eq:6.15} follows from \eqref{eq:6.16}. In the last step we used \eqref{eq:6.7}.

\begin{rem}\label{Corrections} \textbf{Corrections to a preceding paper.}
A few misprints in the
paper \cite{G18}, that were not eliminated during the typesetting, are listed here:  
Page 752, line 24, ``derived from $P^*$'' should be ``derived from
$P^+$''. Page 756, line 9, ``for $p>-1/\mu $'' should be ``for
$p<-1/\mu $''. Page 757, line 8 from below, ``$x_n^a\partial_{n_{\acute
w}}^e(\xi ',0)$'' should be ``$x_n^a\partial_n{\acute w}_e(\xi
',0)$''. Page 762, line 8,  ``$aB$'' should be ``$a^{-1}B$''.
Page 768, line 3, ``Let $P$ is'' should be ``Let $P$
be''; line 13, ``$aB$'' should be ``$a^{-1}B$''. Page 769, line 8 from below, replace ``$\underline uk$'' by ``$\underline u_k$''. 
\end{rem}
 
%\Refs
%\widestnumber\key{[RSV17]}

\end{document}